\documentclass[11pt,a4paper]{amsart}
\usepackage[utf8]{inputenc}
\usepackage{amsmath}
\usepackage{amsfonts}
\usepackage{amssymb}
\usepackage{graphicx}
\usepackage{amsmath,amssymb, amsthm}
\usepackage{pgf}
\usepackage{hyperref}
\usepackage[english]{babel}
\usepackage{mathtools}
\usepackage{caption}
\DeclarePairedDelimiter\floor{\lfloor}{\rfloor}
\theoremstyle{definition}

\theoremstyle{definition}
\newtheorem{rmk}{Remark}
\theoremstyle{definition}

\theoremstyle{definition}\usepackage{amsmath}
\usepackage{amsfonts}
\usepackage{amssymb}
\usepackage{graphicx}
\usepackage{amsmath,amssymb, amsthm}

\theoremstyle{plain}
\newtheorem{thm}{Theorem}
\theoremstyle{plain}

\theoremstyle{plain}

\newcommand{\Z}{\mathbb{Z}}
\newcommand{\Q}{\mathbb{Q}}
\newcommand{\R}{\mathbb{R}}
\newcommand{\C}{\mathbb{C}}

\newcommand{\p}{\mathcal{P}}

\newcommand{\Ok}{\mathcal{O}_K}

\newcommand{\Tr}{\operatorname{Tr}}

\newcommand{\lp}{\left(}
\newcommand{\rp}{\right)}
\newcommand{\N}{\operatorname{N}}

\begin{document}

\address{Dipartimento di Matematica\\
         Università di Milano\\
         via Saldini 50\\
         20133 Milano\\
         Italy}
\title[minimum discriminant of number fields of signature (2,3)]{The minimum discriminant of number fields of degree 8 and signature (2,3)}
\author[F.~Battistoni]{Francesco Battistoni}

\email{francesco.battistoni@unimi.it}



\maketitle

\begin{abstract}
In this paper we describe how to use the algorithmic methods provided by Hunter and Pohst in order to give a complete classification of number fields of degree $8$ and signature $(2,3)$ with absolute discriminant less than a certain bound. The choice of this bound comes from the local corrections given by prime ideals to the lower estimates for discriminants obtained with the Odlyzko-Poitou-Serre method.
\end{abstract}

\section{Introduction}
Let $K$ be a number field of degree $n$, with discriminant $d_K$, and let $r_1$ be the number of real embeddings of $K$ and $r_2$ be the number of couples of complex embeddings, so that $n=r_1+2r_2$.\\
A classical problem asks to establish the minimum value for $|d_K|$ when $K$ ranges in the set of fields with a fixed signature $(r_1,r_2)$. During the last century many methods for answering the question were set: beginning with the classical tools of Geometry of Numbers invented by Minkowski through the analytic estimates involving the Dedekind Zeta functions, due to Odlyzko \cite{odlyzko1990bounds}, Poitou \cite{poitou1977petits} and Serre \cite{serre1986minorations} up to the algorithmic procedures, based on number-geometric ideas, developed by Pohst \cite{pohst1982computation}, Martinet \cite{martinet1985methodes} and Diaz y Diaz \cite{pohst1990minimum} (in collaboration with the previous authors).\\
At the present time, the values for minimum discriminants are known exactly when $n\leq 7$, for any signature, and also when $n=8$, if the signature is either $(8,0)$ or $(0,4)$.\\
For the remaining degrees and signatures, only partial results are known: as an example, if $n=8$ the papers by Cohen, Diaz y Diaz and Olivier \cite{cohen1999tables} and Selmane \cite{selmane1999non} provided minimal values for $|d_K|$, assuming that $K$ is a \emph{non-primitive} number field, i.e. it has subfields different from $\Q$ and $K$. More in detail, the minimum value of $|d_K|$ for a non-primitive number field $K$ with $n=8$ and $(r_1,r_2)=(2,3)$ is equal to $4286875$. These papers also give a complete classification of all non-primitive octic number fields with discriminant less than $6688609$.\\
It has to be remarked that the methods employed in the works aforementioned are such that every possible \emph{primitive} number field with discriminant less than the bound imposed is not considered at all.\\\\
In this paper we deal with the primitive case. Joining the new conclusions
with the already known classification of non-primitive fields we get the following results.
\begin{thm}\label{minimumdisc}
	Let $d_K$ be the discriminant of a number field $K$ with degree $8$ and signature $(2,3)$. Then the minimum value of $|d_K|$ is equal to $4286875$.
\end{thm}

\begin{thm}\label{classification}
There are 56 number fields of degree $8$ and signature $(2,3)$ with $|d_K|\leq 5726300$; with the exception of two non-isomorphic fields with $|d_K|=5365963$, every field in the list is uniquely characterized by the value of $|d_K|$.
\end{thm}
\noindent
Section 2 gives an account of the formulas used to get a lower bound for the discriminant of the type we are interested in. Results descend from manipulations of Weil's explicit formula for Dedekind Zeta functions and from local corrections to the discriminant given by assuming the existence of prime ideals with fixed norm.\\
Section 3 presents the classical Hunter-Pohst Theorem, which is a geometric-number method for bounding the coefficients of a minimum polynomial with functions depending on the discriminant of the number field.\\
In Section 4 the results of the previous lines are combined in order to describe an algorithm which detects all the polynomials with integer coefficients generating
every number field of degree 8 and signature (2,3) with discriminant less in absolute value than $5726300$. \\
Finally, Section 5 gives an account of the results obtained.
\\
The Matlab and PARI/GP programs used for the computation can be found at \\
\url{www.mat.unimi.it/users/battistoni}, together with the data resulting from the implementation of these programs and the complete table of number fields found.\\
On the website one can also find more detailed recalls about the explicit formulas and many technical aspects of the algorithm which are not to be discussed explicitly in this paper. 
\\
\\
I would like to thank my Ph.D Supervisor, prof. Giuseppe Molteni, for having introduced me to this problem, for every useful and clarifying question and for the support given to my efforts. I thank also dr. Lo\"ic Grenié, who gave me many precious suggestions on how to efficiently use PARI/GP, prof. Scherhazade Selmane, for kindly lending me her very useful tables  of local corrections \cite{selmane1999odlyzko}, and prof. Brian Conrey, with whom I had an interesting conversation about this topic.

\section{Estimates and local corrections}\label{Stime n8 r1 2}
The first informations we need for our process come from the study of the so called \textit{explicit formulas} of the Dedekind Zeta function of a number field $K$. This tool has been widely used for the study of minimum discriminants since \cite{poitou1977petits} and here we just recall the most important facts that will be involved in our discussion.\\
Let us define a function
\begin{equation}\label{Tartar}
	f(x):= \lp \frac{3}{x^3}(\sin x -x\cos x)\rp^2
\end{equation}
which is the square of the Fourier transform of the function
\begin{equation*}
	u(x):=\begin{cases}
		1-x^2 & |x|\leq 1 \\
        0 & \text{elsewhere.}
	\end{cases}
\end{equation*}
\noindent
According to \cite{poitou1977petits} and \cite{odlyzko1990bounds}, the function (\ref{Tartar}) was selected by Luc Tartar in 1973 and since then it has been a convenient function for our problem, thanks to the following estimate:

\begin{thm}
Let $K$ be a number field of degree $n$, signature $(r_1,r_2)$ and discriminant $d_K$. Let $\{\p\subset\Ok\}$ denote the set of non zero prime ideals of $\Ok$, and let $y>0$. Finally, let $\gamma$ be Euler's constant. If $f(x)$ is Tartar's function (\ref{Tartar}), then we have the inequality

\begin{equation}\label{Stima}
		\frac{1}{n}\log|d_K| \geq \gamma + \log 4\pi - L_1(y) -\frac{12\pi}{5n\sqrt{y}} + \frac{4}{n}\sum_{\p\subset\Ok}\sum_{m=1}^{\infty}\frac{\log \N(\p)}{1+(\N(\p))^m}f(m\log (\N(\p))\sqrt{y})
\end{equation}
where 

$$L_1(y):=\sum_{k=1}^{\infty}\frac{1}{2k-1}L\lp\frac{y}{(2k-1)^2}\rp + \frac{r_1}{n}\sum_{k=1}^{\infty}(-1)^{n+1}L\lp\frac{y}{k^2}\rp$$
and

\begin{equation*}
L(y):=-\frac{3}{20y^2} + \frac{33}{10y} + 2 + \lp\frac{3}{80y^3}+\frac{3}{4y^2}\rp\lp\log(1+4y)-\frac{1}{\sqrt{y}}\arctan(2\sqrt{y})\rp.
\end{equation*} 
\end{thm}
\begin{proof}
All the details of the proof can be found in \cite{poitou1977petits}. Inequality (\ref{Stima}) is a consequence of the Weil explicit formula applied for the Dedekind Zeta function of $K$, and the choice of Tartar's function allows to turn positive some terms of the formula, obtaining the desired lower bound.
\end{proof}

We want now to use Inequality (\ref{Stima}) in the case of number fields of degree 8 and signature (2,3). 
\\One could try to get estimates which are independent from the terms related to the prime ideals (see \cite{y1980tables}) but, actually, more precise informations follow if one puts some arithmetical conditions.\\
In fact, if we assume that the integer prime numbers have a specified decomposition as products of prime ideals of $K$, then the positive terms related to the prime ideals in (\ref{Stima}) can be directly estimated from below and not simply forgotten; thus one obtains better estimates for $|d_K|$ by algebraic assumptions. These improvements are usually called \emph{local corrections}.
\\
Selmane, in her work \cite{selmane1999odlyzko}, computed the following local corrections to the lower bounds for $|d_K|$, whenever $K$ has $n=8$, $(r_1,r_2)=(2,3)$ and admits a prime ideal $\p$ of norm $\N(\p)$:
\begin{center}
\begin{tabular}{rcc}
&$N(\p)$& $|d_K|>$ \\
\hline
&2 & 11725962\\
&3 & 8336752\\
&4 & 6688609\\
&5 & 5726300\\
&7 & 4682934
\end{tabular}
\end{center}
Consider now a number field $K$ of degree 8 and signature $(2,3)$ with $|d_K|\leq 5726300$. Being this upper bound the local correction given by a prime ideal of norm 5, we immediately get that there are no prime ideals in $K$ with norm between 2 and 5.
This fact obviously implies that any rational integer which is an exact multiple of 2, 3, 4 or 5 cannot be a norm of some element $\alpha\in K$ (we say that $n$ \textit{is an exact multiple of} $q$ if $q|n$ but $q^a \not| n$ for every integer $a\geq 2$).\\
This consideration is of strict importance for the algorithmic procedure described in Section \ref{algoritmo}, p.198.

\section{Hunter-Pohst theorem}
Suppose you have a number field $K$ of degree $n$ and discriminant $d_K$, and let $p(x):=x^n + a_1x^{n-1} +\cdots+a_{n-1}x+a_n$ be the minimum polynomial of a primitive integer $\alpha\in\Ok\setminus\Z$. The following number-geometric methods, which were developed by Pohst in \cite{pohst1982computation}, permit to give a bound for the coefficients of $p(x)$ which depends only on the discriminant of $K$ and the trace of $\alpha$.\\
Recall that the trace and the norm of $\alpha$ are defined as
$$\N(\alpha):=\prod_{i=1}^n \alpha_i \qquad, \qquad \Tr(\alpha):=\sum_{i=1}^n \alpha_i$$
where $m\in\Z$ and $\alpha_i$ is the image of $\alpha$ through the $i$th embedding of $K$ in $\C$.
\\
Define the \emph{symmetric functions}
\begin{equation*}
	S_m:=S_m(\alpha):= \sum_{i=1}^n \alpha_i^m.
\end{equation*}
We know that the following relations occur between the coefficients of $p(x)$ and the symmetric functions:
\begin{align}\label{Congruence}
&a_n=(-1)^n\N(\alpha), \nonumber\\
&S_1=\Tr(\alpha)=-a_1, \nonumber\\
&S_m=-ma_m-\sum_{i=1}^{m-1}a_i S_{m-i}    \text{ }\text{ for } 1<m\leq n, \\
&S_m=-\sum_{i=1}^n a_i S_{m-i} \qquad\text{ for }m>n.\nonumber
\end{align}
If we consider also the symmetric functions with negative exponent we get the following:
\begin{align}
&a_{n-1}=- a_n S_{-1},\nonumber\\
&a_{n-2}=\frac{S_{-1}^2  - S_{-2}}{2}a_n.\nonumber
\end{align} 
The aim is to give a bound for the symmetric functions using the discriminant in order to get a bound for the coefficients.\\
Define now, for every $m\in\Z$, the $m$th \emph{absolute symmetric function}
\begin{equation*}
	T_m:=T_m(\alpha):=\sum_{i=1}^n |\alpha_i|^m.
\end{equation*}
We have the trivial inequality $|S_m|\leq T_m$ and we work with these new functions because they are way easier to study, thanks to the following theorem (which is proved in \cite{pohst1982computation}).

\begin{thm}[Hunter-Pohst]\label{HunterPohst}
Let $K$ be a number field of degree $n$ and discriminant $d$. Then there exists an integer $\alpha\in\Ok\setminus\Z$ which verifies the following inequalities:
\begin{align*}
	&0\leq \Tr(\alpha)\leq \frac{n}{2} \\
    &T_2(\alpha)\leq \frac{(\Tr(\alpha))^2}{n} + \gamma_{n-1}\left|\frac{d}{n}\right|^{1/(n-1)}=:U_2
\end{align*}
where $\gamma_p$ is the $p$-th Hermite constant (the reader who needs its definition can find it in \cite{pohst1997algorithmic}).
\end{thm}

\begin{rmk}
Martinet \cite{martinet1985methodes} gave a useful generalization of this theorem whenever $K$ has a proper subfield of degree $n'$, and this refinement was the key tool used in the classification of non-primitive fields presented in \cite{cohen1999tables} and \cite{selmane1999non}.\\
However, we are now assuming the possible existence of primitive number fields, and so Martinet's theorem cannot be used for that kind of fields. 
\end{rmk}

Theorem \ref{HunterPohst} gives a bound for the trace and the second absolute symmetric function of a certain $\alpha$. The following theorem, which is again due to Pohst and proved in \cite{pohst1982computation}, is crucial for bounding the remaining functions $T_m$.

\begin{thm}\label{ThmMaximum}
Let $T$ and $N$ be two positive constants, and let $n\in\N$ such that $N\leq(T/n)^{n/2}.$ Then, for any $m\in \Z\setminus\{0,2\}$, the function $T_m(x_1,\ldots,x_n):=\sum_{i=1}^n x_i^m$ has an absolute maximum on the compact set

$$\left\{(x_1,\ldots,x_n)\in\R^n\colon \sum_{i=1}^n x_i^2\leq T;\text{ } \prod_{i=1}^nx_i=N; x_i\geq 0 \text{ for }i=1,\ldots,n\right\}$$
and the maximum is attained in a point $(y_1,\ldots,y_n)$ which has at most two different coordinates.
\end{thm}

We briefly describe how to use this theorem in order to bound the functions $T_m$'s, as it is presented in \cite{pohst1982computation} and recalled in \cite{y1987petits}.\\
For every integer value of $t$ between 1 and $n-1$ one looks for the least positive root of the equation
\begin{equation}\label{root}
	t(y^{t-n}N)^{2/t} + (n-t)y^2-T=0.
\end{equation}

Let $u$ be this root. Then, if $\alpha$ satisfies the inequalities of Theorem \ref{HunterPohst}, from Theorem \ref{ThmMaximum} we obtain, for every $m\in\Z\setminus\{0,2\}$, the inequality

\begin{equation}\label{AbsoluteBounds}
T_m(\alpha)\leq \max_{1\leq t<n}\{t(u^{t-n}N)^{m/t} + (n-t)u^m\}=:U_m.
\end{equation}

\section{Description of the algorithm}\label{algoritmo}
We want to detect all the number fields $K$ of degree $8$, signature $(2,3)$ and $|d_K|\leq 5726300$: this will be achieved if we are able to construct all the polynomials of degree 8 which have integer coefficients bounded via the values $U_m$'s found with Theorems \ref{HunterPohst} and \ref{ThmMaximum}.\\
Hunter-Pohst's Theorem guarantees that in this large list of polynomials we will find the minimum polynomial of every primitive number field $K$ we are interested in. It is possible though that this set of polynomials may miss the minimum polynomial of some number field $K$ which has elements $\alpha$'s satisfying the bounds imposed by Theorem \ref{HunterPohst} but generating a proper subfield of $K$. However, this causes no difficulty in our work, because the classification of non-primitive number fields was already done in \cite{cohen1999tables} and \cite{selmane1999non}.\\
The polynomials will be generated ranging the values for the symmetric functions $S_m$'s in the intervals $[-U_m,U_m]$; these value must respect the congruence relations (\ref{Congruence}) and the coefficients we create with these functions must be compatible with the arithmetical conditions imposed in Section \ref{Stime n8 r1 2}.
\\
\begin{itemize}

\item[\textbf{Step 0:}]  Put $n=8$ and choose an integer value for $S_1$ between $0$ and $n/2=4$. Put $a_1=-S_1$.\\
Then compute $U_2$ as in Theorem \ref{HunterPohst} using $|d|=5726300$ (the value of $\gamma_7$ is equal to $64^{1/7}$; see \cite{pohst1997algorithmic}).\\
Next, call $T=U_2$ and compute $(T/n)^{n/2}$ as in the hypothesis of Theorem \ref{ThmMaximum}; choose a positive integer $N\leq (T/n)^{n/2}$ and put either $a_8=N$ or $a_8=-N$ (remember that $N$ is the norm of an element of $K$, and so it must be a value compatible with the conditions of Section \ref{Stime n8 r1 2}). 
This choice is well defined because the inequality between geometric and arithmetic means implies $|a_8|\leq (T/8)^4.$\\
Afterwards, compute the least positive root $u$ of Equation (\ref{root}) and the bounds $U_m$ of Equation (\ref{AbsoluteBounds}), for $m$ between 3 and 8 and $m\in\{-1,-2\}.$ We have now set the intervals $[-U_m,U_m]$ in which the symmetric functions will range.\\
Finally, choose a value $c\in\{0,1\}$ : this is done in order to consider only polynomials $p(x)$ such that $p(1)\equiv c$ mod $2$ in a single run.
\\
\item[\textbf{Step 1:}] Put $S_2$ equal to the maximum integer in $[-U_2,U_2]$ which is equal to $-a_1 S_1$ modulo 2 via Equation (\ref{Congruence}): if $k_2\in\{0,1\}$ is the class of $-a_1 S_1$ modulo 2, then
\begin{equation}\label{S2}
	S_2:= 2\floor*{\frac{U_2-k_2}{2}} + k_2
\end{equation}
and put $a_2:=(-S_2-a_1S_1)/2.$\\
Now, put $S_3$ equal to the maximum value in $[-B_3,B_3]$ which is equal to $-a_1 S_2-a_2 S_1$ modulo 3: in the same way, if $k_3\in\{0,1,2\}$ is the class of $-a_1 S_2-a_2 S_1$ modulo 3, then 
\begin{equation}\label{S3}
	S_3:= 3\floor*{\frac{U_3-k_3}{3}} + k_3
\end{equation}
and we put $a_3:=(-S_3-a_1 S_2-a_2 S_1)/3.$\\
Do the same for $S_4$ up to $S_7$, always respecting the congruence relations and using definitions similar to (\ref{S2}) and (\ref{S3}), and create the coefficients $a_4$ up to $a_7$.\\
Finally call $p(x):=x^8+a_1 x^7 +a_2x^6 +\cdots+a_7x+a_8.$
\\

\item[\textbf{Step 2:}] If the just created $p(x)$ is such that $p(1)\neq c$ mod 2, create a new $p(x)$ by increasing $a_7$ of 1 unit and decrease $S_7$ of 7 units.
\\
\item[\textbf{Step 3:}] In this step of the algorithm one must check if the polynomial $p(x)$ just constructed satisfies the following conditions: $p(x)$ is saved if and only if it satisfies all of them:\\
\begin{itemize}
\item[*] $|p(1)|=|\N(\alpha-1)|\leq ((T-2S_1)/8+1)^4$ and it must be a norm compatible with the requests of Section \ref{Stime n8 r1 2}.

\item[*] $a_7/a_8$, being the number which defines $S_{-1}$, must be in $[-U_{-1},U_{-1}]$. Similarly, $(a_7^2/a_8-2a_6)/a_8$ must be in $[-U_{-2},U_{-2}].$

\item[*] $|p(-1)|=|\N(\alpha+1)|\leq ((T+2S_1)/8+1)^4$ and it must be a norm compatible with the requests of Section \ref{Stime n8 r1 2}.

\item[*] $p(k)=\N(k-\alpha)$ must be an admissible norm, evaluating $k$ from 2 to 5. Similarly for $p(-k)$. 

\item[*] The number $-8a_8-a_7S_1-a_6S_2-a_5S_3-a_4S_4-a_3S_5-a_6S_2-a_1S_7$ defines $S_8$ and so it must belong to $[-U_8,U_8]$.
\\
\end{itemize}

\item[\textbf{Step 4:}] In this step we describe how to move on to the next polynomial.
\\
Suppose we have checked $p(x)$. Then the next polynomial is created by increasing $a_7$ of 2 units, which means that $S_7$ is decreased by 14 units (in this way we don't have to check again the condition on $p(1)$ modulo 2). We now have a new polynomial $p(x)$ that must be tested as described in Step 3.\\
This process of construction and testing is iterated until $S_7$ becomes less than the number
\begin{equation*}
	L_7:=-7\floor*{\frac{U_7-(7-k_7)}{7}}-(7-k_7)
\end{equation*}
which is the smallest number in $[-U_7,U_7]$ equal to $k_7$ modulo 7. If $S_7<L_7$ we delete $a_7$ and $S_7$ and one increases $a_6$ of 1 unit, decreasing $S_6$ of 6 units, go back to Step 1 and create new numbers $S_7$ and $a_7$; then apply again the tests and the increasing process for $a_7$ and $S_7$.\\
The number $S_6$ gets lowered of 6 digits every time we repeat the previous sub-step and the process is iterated until $S_6$ becomes less than 
\begin{equation*}
	L_6:=-6\floor*{\frac{U_6-(6-k_6)}{6}}-(6-k_6).
\end{equation*}
If $S_6<L_6$ then one increases $a_5$ of 1 unit, decreasing $S_5$ of 5 units, and compute new $S_6,a_6,S_7$ and $a_7$.\\
The test is then repeated verifying similar conditions from $S_5$ up to $S_2$: the process terminates once you have $S_2$ less than $L_2+2a_1^2/8$ where
\begin{equation*}
	L_2:=-2\floor*{\frac{U_2-(2-k_2)}{2}}-(2-k_2).
\end{equation*}
\end{itemize}
\noindent 
Once this part of the algorithm is over, we have a list of monic polynomials with integer coefficients and this list depends on the chosen values for $a_1$ and $a_8$ and from the parity of the value at 1 of the polynomials.
\\
\begin{itemize}

\item[\textbf{Step 5:}] The polynomials in the list must be examined again, and $p(x)$ will be displayed, together with the discriminant of the number field $K:=\Q[x]/(p(x))$, if and only if $p(x)$ satisfies each one of the following conditions:
\begin{itemize}
\item[*] $p(x)$ must be irreducible over $\Q$.
\item[*] The discriminant of $p(x)$ must be negative (remember that $r_2=3$ and so the sign of the discriminant is negative)
\item[*] The discriminant of the number field generated by $p(x)$ must be greater than $-5726300$.
\end{itemize}
\end{itemize}

\begin{rmk}
Steps from 0 up to 4 were executed using Matlab. The polynomials survived to these processes were saved in .mat files which were translated in .gp files, so that Step 5 could be run in PARI/GP.
\\
Moreover, the algorithm presents many technical details which have not been discussed here, but can be found in \url{www.mat.unimi.it/users/battistoni} together with the data.
\end{rmk}

\section{Results}
Here is a resume of the results collected from the run of the algorithm above for every case:
\begin{itemize}
\item There are 56 number fields $K$ of signature (2,3) with $|d_K|<5726300$: every field is uniquely determined up to isomorphism by the value of its discriminant, unless $d_K=-5365963$, in which case there are two fields having this discriminant (so that Theorem \ref{classification} is proved).
\item The minimum value of $|d_K|$ for $K$ in the list is 4286875, which was already known as the minimum for non-primitive fields of same degree and signature (and so Theorem \ref{minimumdisc} is proved).

\item There are 46 primitive fields, each one with $G=S_8$, and the minimum value of $|d_K|$ for a field $K$ with $n=8$, $(r_1,r_2)=(2,3)$ and $G=S_8$ is 4296211.

\item Every field found has trivial class group.
\end{itemize}

\begin{rmk}
Actually every field in the list was already known, being contained in the Number Fields Database \url{http://galoisdb.math.upb.de} provided by J\"uergen Kl\"uners and Gunter Malle. Nonetheless, no assumption of completeness for these lists was made, while this work proves that these number fields are in fact the only ones existing in the given range for $|d_K|$ and with signature $(2,3)$.
\end{rmk}

\bibliographystyle{plain}

\end{document}